\documentclass[11pt]{article}
\usepackage{amssymb,amsmath,amsfonts}
\usepackage{graphicx,xcolor,enumitem}
\usepackage{epsfig}
\usepackage{amsthm}
\usepackage{bm}
\usepackage{subcaption}
\usepackage{nicematrix}
\usepackage[round]{natbib}
\usepackage{csquotes}
\renewcommand{\mkbegdispquote}[2]{\itshape}
\usepackage[letterpaper, twoside, hmarginratio=1:1, vmarginratio=1:1, left=1in,top=1in]{geometry}

\RequirePackage[breaklinks=true, hidelinks]{hyperref}
\usepackage{breakcites}

\usepackage{algorithm, algorithmic}

\usepackage{accents}

\newcommand{\cA}{\mathcal{A}}
\newcommand{\cB}{\mathcal{B}}
\newcommand{\E}{\mathbb{E}}
\newcommand{\R}{\mathbb{R}}

\newcommand{\cL}{{\mathcal L}}

\newcommand{\id}{{\mathbf 1}}

\newcommand{\cY}{{\mathcal Y}}
\newcommand{\cX}{{\mathcal X}}

\newtheorem{theorem}{Theorem}

\newtheorem{assumption}[theorem]{Assumption}

\newtheorem{lemma}[theorem]{Lemma}

\newtheorem{proposition}[theorem]{Proposition}
\newtheorem{remark}[theorem]{Remark}

\theoremstyle{definition}

\numberwithin{equation}{section}
\numberwithin{theorem}{section}

\begin{document}

\title{Existence of Markov equilibrium control in discrete time}

\author{
	Erhan Bayraktar \thanks{Department of Mathematics, University of Michigan, Ann Arbor, Email: erhan@umich.edu. Erhan Bayraktar is partially supported by the National Science Foundation under grant DMS-2106556 and by the Susan M. Smith chair.}
	\and Bingyan Han \thanks{Department of Mathematics, University of Michigan, Ann Arbor. Email: byhan@umich.edu.}
}

\date{December 8, 2023}
\maketitle

\begin{abstract}
	For time-inconsistent stochastic controls in discrete time and finite horizon, an open problem in Bj\"ork and Murgoci (Finance Stoch, 2014) is the existence of an equilibrium control. A nonrandomized Borel measurable Markov equilibrium policy exists if the objective is inf-compact in every time step. We provide a sufficient condition for the inf-compactness and thus existence, with costs that are lower semicontinuous (l.s.c.) and bounded from below and transition kernels that are continuous in controls under given states. The control spaces need not to be compact.
	\\[2ex] 
	\\[2ex]
\end{abstract}

\section{Introduction} 
In this short note, we revisit discrete-time stochastic controls with time-inconsistent objectives studied in \cite{bjork14FS}. Unlike standard stochastic optimal control problems, time inconsistency reflects the phenomenon that agent preferences change in time. Bellman optimality principle fails. The optimal plan at the present moment is not followed in the future. Classical examples include non-exponential discounting \citep{strotz1955myopia}, prospect theory \citep{kahneman1979prospect}, mean-variance portfolio selection \citep{basak10}, and optimal stopping \citep{bayraktar2021equilibrium}. See the recent monograph \cite{bjork2021time} and reference therein for a comprehensive review.  

As proposed in \cite{strotz1955myopia}, agents should only consider the policy that they can actually follow. By the consistent planning framework in \cite{strotz1955myopia}, agents acknowledging time inconsistency view themselves at different future times as distinct planners. In this context, an equilibrium policy $\pi^* = (u^*_1, \ldots, u^*_{T-1})$ should satisfy that, if all future selves at time $t+1, \ldots, T-1$ use $u^*_{t+1}, \ldots, u^*_{T-1}$, then it is optimal for the planner at time $t$ to use $u^*_t$. Thus, there is no incentive for the planner to deviate from $\pi^*$ at any time $t$. A formal definition is given in \citet[Definition 2.5]{bjork14FS}.

If equilibrium policies exist, \citet[Theorem 3.14]{bjork14FS} characterizes them with the extended Bellman equation. However, the existence of an equilibrium policy remains open \cite[Section 6]{bjork14FS}. Recall the Borel space framework in \citet[Chapters 7 and 8]{bertsekas1978stoch} for the classical stochastic optimal control. The existence of optimal policies is established through measurable selection theorems. The optimal policy obtained usually has weaker properties compared with the objective function. For instance, if the objective is l.s.c., then an optimal policy, if it exists, is Borel measurable \cite[Proposition 7.33]{bertsekas1978stoch}. If the objective is lower semianalytic, then an $\varepsilon$-optimal policy is universally measurable \cite[Proposition 7.50]{bertsekas1978stoch}. However, the extended Bellman equation introduces auxiliary functions \eqref{eq:bfh}, tied to the optimal policy, and thus inheriting the weaker properties. Consequently, demonstrating the existence of an equilibrium policy via a recursive application of measurable selection becomes challenging.

The key observation is, if the transition kernels of states have nice properties, then it may still be possible to apply certain measurable selection theorem repeatedly. The weak continuity of transition kernels is often inadequate. But a stronger continuity condition, given in Assumption \ref{assum} (4), is sufficient for the existence of an equilibrium policy. Notably, Assumption \ref{assum} (4) is not new and has been considered in \cite{hernandez1991scl} on average cost Markov decision processes. It is relevant to setwise convergence and also used in \citet[Assumption 8.3.1 (c)]{hernandez1999} and \citet[Assumption 3.1 (b)]{saldi2018finite}. In conjunction with a version of the Borel measurable selection theorem \cite[Theorem 4.1 and Corollary 4.3]{rieder1978measurable} under the inf-compact condition, we managed to prove the existence of a nonrandomized Borel measurable Markov equilibrium policy.

In the literature, \cite{jaskiewicz2021markov} explored infinite-horizon Markov decision processes with quasi-hyperbolic discounting in their equation (2.2). The existence of a randomized stationary Markov perfect equilibrium is proved with continuous costs and transition kernels with densities; see \citet[Assumptions C3.1 -- C3.2]{jaskiewicz2021markov}. While \citet[Assumption C3.2]{jaskiewicz2021markov} closely resembles Assumption \ref{assum} (4), their proof differs since they assumed compactness instead of inf-compactness and relied on the structure of quasi-hyperbolic discounting to apply the fixed point theorem. Other existence results are available for different formulations. \cite{huang2021strong} considered weak and strong equilibria in the continuous time under certain convexity (concavity) conditions. \cite{bayraktar2021equilibrium} and reference therein studied time-inconsistent stopping problems. \cite{bayraktar2023relaxed} examined relaxed equilibria with regularization. \cite{bayraktar2023equilibrium} considered optimal transport with time-inconsistency and provided the existence of the so-called equilibrium transport.

The rest of the paper is organized as follows. Section \ref{sec:form} formulates the general problem, while readers need to refer to \cite{bjork14FS} for the derivation of the extended Bellman equation. Section \ref{sec:exist} establishes the sufficient condition for the existence. Illustrative examples are in Section \ref{sec:example}.

\section{Problem formulation}\label{sec:form}
In this work, a topological space $\cX$ will always be endowed with the Borel $\sigma$-algebra $\cB(\cX)$. A Borel subset of a Polish (complete separable metric) space is called a Borel space. Consider a nonstationary discrete-time controlled Markov process with the following specifications: 
\begin{enumerate}[label={(\arabic*)}]
	\item The finite horizon $T$: A positive integer.
	\item State spaces $\cX_t$, $t \in \{1, \ldots, T\}$: $\cX_t$ is a nonempty Borel space.
	\item Control spaces $\cA_t$, $t \in \{1, \ldots, T-1\}$: $\cA_t$ is a nonempty Borel space. $u_t \in \cA_t$ is a control chosen at time $t$.
	\item Control constraints $U_t$, $t \in \{1, \ldots, T-1\}$: $U_t$ is a function from $\cX_t$ to the set of nonempty subsets of $\cA_t$. Suppose the set
	\begin{equation}
		\Gamma_t := \left\{ (x_t, u_t) | x_t \in \cX_t, u_t \in U_t(x_t) \right\}
	\end{equation}
	is a Borel subset of $\cX_t \times \cA_t$.
	\item Transition kernels $P_t(dx_{t+1}| x_t, u_t)$, $t \in \{1, \ldots, T-1\}$: Suppose $P_t(dx_{t+1}| x_t, u_t)$ is a Borel-measurable stochastic kernel on $\cX_{t+1}$ given $\Gamma_t$. 
\end{enumerate}

The system begins at some state $x_1 \in \cX_1$ and proceeds successively to state spaces $\cX_2, \ldots, \cX_{T-1}$ and finally terminates in $\cX_T$. A policy controlling the system is a sequence $\pi_{t:T-1} := (\pi_t, \ldots, \pi_{T-1})$, where $\pi_t(\cdot|x_{1:t}, u_{1:t-1})$ is a Borel measurable stochastic kernel on $\cA_t$ given the history and it satisfies $\pi_t(U_t(x_t) | x_{1:t}, u_{1:t-1}) = 1$ for every $(x_{1:t}, u_{1:t-1})$. In this work, we consider $\pi_t(\cdot|x_{1:t}, u_{1:t-1})$ that depends only on the current state and time and concentrates at one point for each state $x_t$, which is called nonrandomized Markov policies. Thus, $\pi_t(\cdot|x_{1:t}, u_{1:t-1}) = u_t(x_t)$ for some Borel measurable functions $u_t(\cdot)$. We still use the notation $\pi^t = (u_t, \ldots, u_{T-1})$.

At time $t$, the agent wants to minimize a state-dependent and nonlinear cost objective \cite[Equation 3.8]{bjork14FS}:
\begin{align}
	J_t(x; \pi^t) :=& \int \left[ \sum^{T-1}_{k=t} C_k (t, x, x_k, u_k) + F(t, x, x_T) \right] P(dx_{t+1:T}| x, \pi^t) \nonumber \\
	& + G\left(t, x, \int H(x_T) P(dx_T| x, \pi^t) \right), \label{eq:obj}
\end{align}
where $P(dx_{t+1:T}| x, \pi^t) := P_{T-1}(dx_T| x_{T-1}, u_{T-1}(x_{T-1})) \ldots P_t(dx_{t+1}| x_t, u_t(x_t))$ under the policy $\pi^t$ and $x_t = x$. The costs $C_k(s, y, x_k, u_k)$, $F(s, y, x_T)$, $H(x_T)$, and $G(s, y, h)$ are Borel measurable. Given $t \in \{1, \ldots, T-1\}$, the term $(s, y)$ with $s \in \{1, \ldots, t\}$ and $y \in \cX_s$ represents the state-dependence. For example, consider $F(t, x, x_T)$. At a different time $s \neq t$, it becomes $F(s, x_s, x_T)$ and thus the preference of the agent changes.  $G(s, y, h)$ is nonlinear and does not satisfy the tower property in general. Therefore, Bellman optimality principle fails due to the state-dependence and nonlinearity. We refer readers to \cite{bjork14FS,bjork2021time} for specific applications with time inconsistency.  

Inspired by the consistent planning in \cite{strotz1955myopia}, \citet[Definition 2.5]{bjork14FS} introduced a subgame perfect Nash equilibrium strategy as follows. Given time $t<T$ and state $x_t = x$, define a control policy $\pi^{t, u} = (u, u^*_{t+1}, \ldots, u^*_{T-1})$ on $t, \ldots, T-1$, where $u \in U_t(x_t)$ is an arbitrary control value and $u^*_{s} = u^*_{s}(x_s), \forall \, x_s \in \cX_s, \, s \in\{t+1, \ldots, T-1\},$ are Borel measurable functions. A Borel measurable function $u^*_t(x_t)$ is a nonrandomized Markov equilibrium control at time $t$ if for every $x$, we have
\begin{equation*}
	\inf_{u \in U_t(x_t)} J_t(x; \pi^{t, u}) = J_t(x; \pi^{t, *}),
\end{equation*}
where $\pi^{t, *} = (u^*_t, u^*_{t+1}, \ldots, u^*_{T-1})$. Denote the corresponding equilibrium value function as $V_t(x) = J_t(x; \pi^{t, *})$.

If $\pi^* = (u^*_1, \ldots, u^*_{T-1})$ exists, \citet[Theorem 3.14]{bjork14FS} proved that the value function $V_t$ satisfies the extended Bellman equation given by
{\allowdisplaybreaks
	\begin{align}
	V_t(x_t) = \inf_{u \in U_t(x_t)} \Big\{&  C_t(t, x_t, x_t, u) + \int V_{t+1} (x_{t+1}) P_t(dx_{t+1}| x_t, u) \nonumber \\
	& + \sum^{T-1}_{k=t+1} \int b_k (t, x_t, x_{t+1}) P_t(dx_{t+1}| x_t, u)  \label{eq:original} \\
	& - \sum^{T-1}_{k=t+1} \int b_k (t+1, x_{t+1}, x_{t+1}) P_t(dx_{t+1}| x_t, u) \nonumber \\
	& + \int f(t, x_t, x_{t+1})  P_t(dx_{t+1}| x_t, u) - \int f(t+1, x_{t+1}, x_{t+1})  P_t(dx_{t+1}| x_t, u) \nonumber \\
	 & + G \left( t, x_t, \int h(x_{t+1}) P_t(dx_{t+1}| x_t, u) \right) \nonumber \\
	 & - \int G \left(t+1, x_{t+1}, h(x_{t+1}) \right) P_t(dx_{t+1}| x_t, u) \Big\}. \nonumber 
\end{align} 
}
The functions $b_k$, $f$, and $h$ are defined backwardly as
\begin{equation}\label{eq:bfh}
	\begin{aligned}
			b_k(s, y, x_{t+1}) &:=&& \int C_k (s, y, x_k, u^*_k(x_k)) P(dx_k | x_{t+1}, \pi^{t+1, *}), \quad k \in \{t+1, \ldots, T-1\}, \\
		f(s, y, x_{t+1}) &:=&& \int F(s, y, x_T) P(dx_T| x_{t+1}, \pi^{t+1, *}), \\
		h(x_{t+1}) & := && \int H(x_T) P(dx_T| x_{t+1}, \pi^{t+1, *}),
	\end{aligned}
\end{equation} 
with a nonrandomized Markov equilibrium policy $\pi^{t+1, *} = (u^*_{t+1}, \ldots, u^*_{T-1})$ found for time $t+1, \ldots, T-1$.

In fact, the value function $V_{t+1}$ satisfies
\begin{equation}
	V_{t+1}(x_{t+1}) = \sum^{T-1}_{k=t+1} b_k (t+1, x_{t+1}, x_{t+1}) + f(t+1, x_{t+1}, x_{t+1}) + G(t+1, x_{t+1}, h(x_{t+1})).
\end{equation}
Thus, the extended Bellman equation \eqref{eq:original} reduces to
\begin{equation}\label{eq:prob} 
	V_t(x_t) = \inf_{u \in U_t(x_t)} \cL_{t+1}(t, x_t, x_t, u),
\end{equation}
where 
\begin{align*}
	\cL_{t+1}(s, y, x_t, u) := &  C_t(s, y, x_t, u) + \sum^{T-1}_{k=t+1} \int b_k (s, y, x_{t+1}) P_t(dx_{t+1}| x_t, u) \\
	& + \int f(s, y, x_{t+1})  P_t(dx_{t+1}| x_t, u) + G\left(s, y, \int h(x_{t+1}) P_t(dx_{t+1}| x_t, u) \right). 
\end{align*}
We emphasize that the extended Bellman equation \eqref{eq:prob} relies on the auxiliary functions $b_k$, $f$, and $h$ in \eqref{eq:bfh} backwardly. When the existence of a nonrandomized Markov equilibrium policy $\pi^{t+1, *}$ is proved for time $t+1, \ldots, T-1$, then the characterization of $u^*_t$ is given by \eqref{eq:prob} at time $t$.

\section{Existence}\label{sec:exist}
In Lemma \ref{lem:select}, we quote a Borel measurable selection theorem in \citet[Theorem 4.1 and Corollary 4.3]{rieder1978measurable}, which is relevant to \citet[Corollary 1]{brown1973}. Let $D$ be a subset of a product space $\cX \times \cY$. Denote ${\rm proj}(D)$ as the projection of $D$ onto $\cX$. For $x \in \cX$, the $x$-section of $D$ is denoted by $D_x$. Crucially, Lemma \ref{lem:select} does not require the compactness of $D_x$ for the existence of a minimizer. Instead, it supposes the level sets in \eqref{eq:inf_compact} are compact. Functions satisfying this property are called inf-compact, see \citet[p.28]{hernandez1996}.   
\begin{lemma}{\cite[Theorem 4.1 and Corollary 4.3]{rieder1978measurable}}\label{lem:select}
	Let $\cX$ and $\cY$ be two Borel spaces. Consider a Borel subset $D$ of $\cX \times \cY$ and a Borel measurable function $f(x, y): D \rightarrow \R$. If the set
	\begin{equation}\label{eq:inf_compact}
		\{ y \in D_x | f(x, y) \leq r \}
	\end{equation} 
	is compact for every $r \in \R$ and $x \in {\rm proj}(D)$, then there exists a Borel measurable minimizer $g: {\rm proj}(D) \rightarrow \cY$, such that
	\begin{equation}
		g(x) \in D_x \text{ and } f(x, g(x)) = \inf_{y \in D_x} f(x, y), \quad x \in {\rm proj}(D).
	\end{equation} 
\end{lemma}
\begin{remark}
	Since compact subsets of metric spaces are closed, the inf-compact assumption in Lemma \ref{lem:select} implies that $f(x, \cdot)$ is l.s.c. with each $x \in {\rm proj}(D)$. Moreover, since $D_x = \cup_{r \in \mathbb{Z}} \{ y \in D_x | f(x, y) \leq r \}$, $D_x$ is $\sigma$-compact with each $x \in {\rm proj}(D)$. Note that \citet[Corollary 1]{brown1973} also holds for Borel spaces \cite[Remark 2]{brown1973}. Hence, the assumptions in \citet[Corollary 1]{brown1973} are satisfied under the assumptions in Lemma \ref{lem:select}.
\end{remark}

If we can verify that
\begin{equation}\label{eq:Linf}
	\{ u \in U_t(x_t) | \cL_{t+1}(t, x_t, x_t, u) \leq r \} 
\end{equation}
is compact for every $r \in \R$ and $x_t \in \cX_t$, then a Borel measurable minimizer $u^*_t(x_t)$ exists for \eqref{eq:prob} at time $t$. If the inf-compact condition \eqref{eq:Linf} can be verified for every time $t$ after the existence of an equilibrium is proved for time $t+1, \ldots, T-1$, then a nonrandomized Markov equilibrium policy exists. Assumption \ref{assum} provides a sufficient condition.
\begin{assumption}\label{assum}
	\begin{enumerate}[label={(\arabic*)}]
		\item The functions $F, G, H$, and $C_t$, $t \in \{1, \ldots, T-1\}$ in the objective \eqref{eq:obj} are Borel measurable, real-valued, and nonnegative.
		\item For every $t \in \{ 1, \ldots, T-1\}$, $C_t (s, y, x_t, u)$ is l.s.c. in $u$ for each $(s, y, x_t)$, where $s \in \{1, \ldots, t\}$, $y \in \cX_s$, and $x_t \in \cX_t$. Moreover, the set
		\begin{equation}\label{eq:lowerset}
			U_t(r, x_t) := \{ u \in U_t(x_t) | C_t(t, x_t, x_t, u) \leq r \}
		\end{equation}
		is compact for every $r \in \R$, $x_t \in \cX_t$, and $t\in \{1, \ldots, T-1\}$.
		\item $G(s, y, h)$ is l.s.c. in $h$ and nondecreasing in $h$ for each $(s, y)$, where $s \in \{1, \ldots, T-1\}$ and $y \in \cX_s$.
		\item For each $t \in \{ 1, \ldots, T-1\}$, $\int V(x_{t+1}) P_t(dx_{t+1}| x_t, u)$ is continuous in $u$ for each $x_t \in \cX_t$ and bounded Borel measurable function $V$. 
	\end{enumerate}
\end{assumption}
\begin{remark}
		Assumption \ref{assum} (4) is equivalent to imposing that $\int V(x_{t+1}) P_t(dx_{t+1}| x_t, u)$ is l.s.c. in $u$ for each $x_t \in \cX_t$ and bounded Borel measurable function $V$. We could also consider nonnegative $V$ only; see \citet[p. 44]{hernandez1999}. Assumption \ref{assum} (3) is a sufficient condition that the composition $G\left(s, y, \int h(x_{t+1}) P_t(dx_{t+1}| x_t, u) \right)$ is l.s.c. in $u$ when $(s, y, x_{t})$ is fixed. However, for a special case in Section \ref{sec:MVP}, the nondecreasing property is not needed and the composition is even continuous.
\end{remark}

An immediate consequence of Assumption \ref{assum} (4) is given by Lemma \ref{lem:lsc}.
\begin{lemma}\label{lem:lsc}
	Suppose Assumption \ref{assum} (4) holds, then $\int V(x_{t+1}) P_t(dx_{t+1}| x_t, u)$ is l.s.c. in $u$ for each $x_t \in \cX_t$ and Borel measurable function $V \geq 0$. 
\end{lemma}
\begin{proof}
	For a nonnegative Borel measurable function $V$, there exists a sequence $V_k$ of Borel measurable and bounded functions converging increasingly to $V$. By monotone convergence theorem,
	\begin{align*}
		\int V(x_{t+1}) P_t(dx_{t+1}| x_t, u) = \int \sup_k V_k(x_{t+1}) P_t(dx_{t+1}| x_t, u) = \sup_k \int V_k(x_{t+1}) P_t(dx_{t+1}| x_t, u),
	\end{align*}
	with a fixed $x_t$. Under Assumption \ref{assum} (4), it is a supremum of continuous functions in $u$. Thus, it is l.s.c. in $u$ by \citet[Box 1.5, p.6]{santambrogio2015optimal}.
\end{proof}

Assumption \ref{assum} (4) is crucial and specifies the topology needed for the proof of the existence Theorem \ref{thm:exist}. The usual weak topology considers bounded and continuous functions $V$, which is not enough to apply the Borel measurable selection theorem \ref{lem:select} repeatedly. Assumption \ref{assum} (4) implies the setwise convergence. It is widely used in the literature of Markov decision processes with average costs; see \citet[Condition (S)(2)]{schal1993average}, \citet[Assumption 2.1 (b)]{hernandez1991scl}, \citet[Assumption 8.3.1 (c)]{hernandez1999}, and \citet[Assumption 3.1 (b)]{saldi2018finite}. 
\begin{theorem}\label{thm:exist}
	Suppose Assumption \ref{assum} holds, then there exists a nonrandomized Markov equilibrium policy $\pi^* = (u^*_1, \ldots, u^*_{T-1})$.
\end{theorem}
\begin{proof}
	We prove the claim by backward induction. At time $T-1$, the objective function is 
	\begin{align*}
	\cL_T (T-1, x_{T-1}, x_{T-1}, u) = &  C_{T-1}(T-1, x_{T-1}, x_{T-1}, u)  \\
	& + \int F(T-1, x_{T-1}, x_{T})  P_{T-1}(dx_T| x_{T-1}, u) \\
	& + G\left(T-1, x_{T-1}, \int H(x_T) P_{T-1}(dx_T| x_{T-1}, u) \right). 
	\end{align*}
	By  Lemma \ref{lem:lsc}, $\int F(T-1, x_{T-1}, x_{T})  P_{T-1}(dx_T| x_{T-1}, u)$ and $\int H(x_T) P_{T-1}(dx_T| x_{T-1}, u)$ are l.s.c. in $u$ for each $x_{T-1}$. Since $G(T-1, x_{T-1}, \cdot)$ is l.s.c. and nondecreasing, the composition $G(T-1, x_{T-1},$ $ \int H(x_T) P_{T-1}(dx_T| x_{T-1}, u))$ is l.s.c. in $u$ when $x_{T-1}$ is fixed. Hence, the objective $\cL_T (T-1, x_{T-1}, x_{T-1}, u)$ is l.s.c. in $u$ for each $x_{T-1}$.
	
	Since $F, G \geq 0$, $\cL_T (T-1, x_{T-1}, x_{T-1}, u) \leq r$ implies that $C_{T-1}(T-1, x_{T-1}, x_{T-1}, u) \leq r$. Therefore,
	\begin{align}\label{eq:subset}
			\{ u \in U_{T-1}(x_{T-1}) | \cL_T(T-1, x_{T-1}, x_{T-1}, u) \leq r \} 
	\end{align} 
	is a subset of the compact set $U_{T-1}(r, x_{T-1})$ defined in \eqref{eq:lowerset}. Moreover, the subset in \eqref{eq:subset} is closed by the definition of lower semicontinuity. Since closed subsets of compact sets are compact, we have verified the inf-compact condition \eqref{eq:inf_compact}. By Lemma \ref{lem:select}, there exists a Borel measurable minimizer $u^*_{T-1}(x_{T-1})$ for 
	\begin{equation*}
		V_{T-1} (x_{T-1}) = \inf_{u \in U_{T-1}(x_{T-1})} \cL_T (T-1, x_{T-1}, x_{T-1}, u). 
	\end{equation*}
	Moreover, 
	\begin{align*}
		b_{T-1}(s, y, x_{T-1}) &:= C_{T-1} (s, y, x_{T-1}, u^*_{T-1}(x_{T-1})), \\
		f(s, y, x_{T-1}) &:= \int F(s, y, x_T) P_{T-1}(dx_T| x_{T-1}, u^*_{T-1}(x_{T-1})), \\
		h(x_{T-1}) & :=  \int H(x_T) P_{T-1}(dx_T| x_{T-1}, u^*_{T-1}(x_{T-1})),
	\end{align*}
	and $V_{T-1} (x_{T-1})$ are nonnegative Borel measurable functions.
	
	At time $T-2$, the objective is
	\begin{align*}
		\cL_{T-1}(T-2, x_{T-2}, x_{T-2}, u) = &  C_{T-2}(T-2, x_{T-2}, x_{T-2}, u) \\
		& + \int b_{T-1}(T-2, x_{T-2}, x_{T-1})  P_{T-2}(dx_{T-1}| x_{T-2}, u) \\
		& + \int f(T-2, x_{T-2}, x_{T-1})  P_{T-2}(dx_{T-1}| x_{T-2}, u) \\
		& + G\left(T-2, x_{T-2}, \int h(x_{T-1}) P_{T-2}(dx_{T-1}| x_{T-2}, u) \right). 
	\end{align*}
	Similarly, $\cL_{T-1}(T-2, x_{T-2}, x_{T-2}, u)$ is l.s.c. in $u$ for each $x_{T-2}$. Hence, we can prove the stated claim backwardly with Lemma \ref{lem:select}.
\end{proof}

Lemma \ref{lem:cont} gives a sufficient condition for Assumption \ref{assum} (4). The proof is similar to \citet[Lemma 4.1]{feinberg2007optimality}, while some calculation details are different. We include the proof in the Appendix for the completeness of the paper.  
\begin{lemma}\label{lem:cont}
	 Let the state spaces $\cX_t = \R$. Consider the transition kernels $P_t(dx_{t+1}| x_t, u)$ specified by the following system:
	\begin{equation}
		x_{t+1} = \mu(x_t, u_t) + \sigma(x_t, u_t) W_t, \quad t \in \{1, \ldots, T-1\}.
	\end{equation}
	Suppose
	\begin{enumerate}[label={(\arabic*)}]
		\item $\{ W_t \}_{t=1}^{T-1}$ are i.i.d. random noises with a Borel density function $p(w)$, $w \in (-\infty, \infty)$, with respect to the Lebesgue measure.
		\item The Borel measurable functions $\mu(x_t, u)$ and $\sigma(x_t, u)$ are continuous in $u$ for each $x_t \in \R$. 
		\item There exists a constant $\underline{\sigma} > 0$ such that $\sigma(x_t, u) \geq \underline{\sigma}$.
	\end{enumerate}
	Then $\int V(x_{t+1}) P_t(dx_{t+1}| x_t, u)$ is continuous in $u$ for each $x_t \in \R$ and bounded Borel measurable function $V$. 
\end{lemma}

\begin{remark}
	\begin{enumerate}[label={(\arabic*)}]
		\item If we assume the density $p(w)$ of random noises is continuous, then Lemma \ref{lem:cont} follows from Scheff\'e's Theorem \cite[p.117]{durrett2019probability}.
		\item The assumption  $\sigma(x_t, u) \geq \underline{\sigma} > 0$ is needed. For example, consider the state following $x_{t+1} = u_t W_t$, where $W_t \sim N(0, 1)$ is a standard normal random variable. The Borel measurable function $V$ satisfies $ V(0) = 1$ and $V(x) = 0$ if $x \neq 0$. Then $\int V(x_{t+1}) P_t(dx_{t+1}| x_t, u) = \int V(uw) p(w) dw = 1$ if $u=0$ and $0$ if $u \neq 0$. Hence, it is not l.s.c. at $u=0$.
	\end{enumerate}
\end{remark}

\section{Examples}\label{sec:example}
\subsection{Time-inconsistent linear-quadratic regulator}
Consider a time-inconsistent linear-quadratic regulator in \citet[Section 9.4]{bjork14FS}. The state-dependent cost objective is given by
\begin{align}\label{example1}
	J_t(x; \pi^t) =& \int \Big[ \sum^{T-1}_{k=t} u^2_k \Big] P(dx_{t+1:T}| x, \pi^t) +  \int  (x_T - x)^2 P(dx_{T}| x, \pi^t),
\end{align}
with a scalar state following
\begin{equation*}
	x_{t+1} = a x_t + b u_t + \sigma W_{t}, 
\end{equation*}
where $a$ and $b$ are known constants. The random noises $\{W_{t}\}^{T-1}_{t=1}$ are i.i.d. and follow standard normal distribution $N(0, 1)$. By Lemma \ref{lem:cont}, Assumption \ref{assum} (4) is valid. It is also direct to verify other conditions in Assumption \ref{assum}. By Theorem \ref{thm:exist}, a nonrandomized Markov equilibrium policy exists. The explicit solution is given in \citet[Section 9.4]{bjork14FS}. 

Besides the objective \eqref{example1}, we can also consider nonlinear cost functionals satisfying Assumption \ref{assum}, such as 
\begin{align*}
	J_t(x; \pi^t) =& \int \Big[ \sum^{T-1}_{k=t} u^2_k \Big] P(dx_{t+1:T}| x, \pi^t) +  \left( \int  \max \left\{x_T , 0 \right\} P(dx_{T}| x, \pi^t) \right)^2,
\end{align*}
since $h^2$ is a continuous and nondecreasing function on $[0, \infty)$.

\subsection{Mean-variance portfolio selection}\label{sec:MVP}
Assumption \ref{assum} is only a sufficient condition for the existence. In fact, if we can verify that the objective in \eqref{eq:prob} is inf-compact backwardly, then the existence holds. 

Consider the mean-variance portfolio selection with the objective
\begin{equation}\label{example3}
	J_t(x; \pi^t) = \text{Var}_t[x_T] - \gamma \E_t[x_T] = \int (x^2_T - \gamma x_T) P(dx_{T}| x, \pi^t) - \left( \int x_T P(dx_{T}| x, \pi^t) \right)^2
\end{equation}
and the wealth (state) process
\begin{equation}\label{eq:wealth}
	x_{t+1} = R x_t + u_t Z_t.
\end{equation}
$R \geq 1$ is a known constant. $\{Z_t\}^{T-1}_{t=1}$ are i.i.d. random variables with mean $\mu$ and variance $\sigma^2 > 0$.

With our notations, the objective \eqref{example3} corresponds to
\begin{equation*}
	F(x_T) = x^2_T - \gamma x_T, \quad H(x_T) = x_T, \quad G(h) = - h^2.
\end{equation*}

While Assumption \ref{assum} is not satisfied, we have the following proposition for the inf-compactness and the existence.
\begin{proposition}
	For the objective \eqref{example3} with the state \eqref{eq:wealth}, the level set \eqref{eq:Linf} is inf-compact and thus a nonrandomized Markov equilibrium policy exists. 
\end{proposition}

\begin{proof}

At time $T-1$, it is direct to show the objective is quadratic in $u$ and an equilibrium control $u^*_{T-1}$ is a constant independent of $x_{T-1}$. To apply backward induction, we assume the equilibrium controls $\{u^*_{t+1}, \ldots, u^*_{T-1} \}$ exist and are constant. Then we calculate the objective at time $t$. The wealth process yields
\begin{equation*}
	x_T = R^{T-t-1} x_{t+1} + \sum_{k=t+1}^{T-1} R^{T-1-k}u^*_k Z_k.
\end{equation*}
 Therefore,
 \begin{equation*}
 	h(x_{t+1}) = \int x_T P(dx_T| x_{t+1}, \pi^{t+1,*}) = R^{T-t-1} x_{t+1} + \sum_{k=t+1}^{T-1} R^{T-1-k} \mu u^*_k =:   R^{T-t-1} x_{t+1} + h^0_{t+1}.
 \end{equation*}
\begin{align*}
	 \int x^2_T P(dx_T| x_{t+1}, \pi^{t+1,*}) = & R^{2(T-t-1)} x^2_{t+1} + 2 R^{T-t-1} x_{t+1}\sum_{k=t+1}^{T-1} R^{T-1-k} \mu u^*_k \\
	& +  \sum_{\substack{k \neq l, \\ k,l \in \{t+1, \ldots, T-1\}} } R^{T-1-k} R^{T-1-l} \mu^2 u^*_k u^*_l + \sum^{T-1}_{k=t+1} R^{2(T-1-k)} (u^*_k)^2 (\mu^2 + \sigma^2) \\
	=: & R^{2(T-t-1)} x^2_{t+1} + a^1_{t+1} x_{t+1} + a^0_{t+1}.
\end{align*}
At time $t$, we have $x_{t+1} = R x_t + u Z_t$, with the candidate equilibrium control $u$ to be solved. We note that
\begin{align*}
	& \int f(x_{t+1})  P_t(dx_{t+1}| x_t, u) \\
	 & \qquad = \int \big[ R^{2(T-t-1)} x^2_{t+1} + a^1_{t+1} x_{t+1} + a^0_{t+1} - \gamma (R^{T-t-1} x_{t+1} + h^0_{t+1}) \big] P_t(dx_{t+1}| x_t, u) \\
	 & \qquad = R^{2(T-t-1)} \big[R^2 x^2_t + 2 R x_t \mu u + (\mu^2 + \sigma^2)u^2 \big] + ( a^1_{t+1} - \gamma R^{T-t-1})(Rx_t + \mu u) \\
	 & \qquad \quad + a^0_{t+1} - \gamma h^0_{t+1}.
\end{align*} 
\begin{align*}
G\left(\int h(x_{t+1}) P_t(dx_{t+1}| x_t, u) \right) = - (R^{T-t}x_t + R^{T-t-1}\mu u + h^0_{t+1})^2.
\end{align*} 
Clearly, the objective $\cL_{t+1}(x_t, u) = \int f(x_{t+1})  P_t(dx_{t+1}| x_t, u) + G\left( \int h(x_{t+1}) P_t(dx_{t+1}| x_t, u) \right)$ at time $t$ is a quadratic function of $u$. Since the coefficient of $u^2$ is $R^{2(T-t-1)}\sigma^2 > 0$, the level set \eqref{eq:Linf} is compact. Thus, an equilibrium control $u^*_t$ at time $t$ exists. Indeed, $u^*_t$ is also constant, since coefficients of $u^2$ and $u$ do not rely on $x_t$. Then we can prove the existence by backward induction.
\end{proof} 

The existence of an equilibrium strategy with the so-called state-dependent risk aversion \cite[Section 9.2]{bjork14FS} can be proved similarly. Indeed, these two cases have explicit equilibrium policies.

\subsection{Non-exponential discounting}
Consider the portfolio selection with non-exponential discounting and the exponential utility:
\begin{equation}\label{example4}
	J_t(x; \pi^t) = \int \varphi(T - t) \frac{e^{-\gamma x_T}}{\gamma} P(dx_{T}| x, \pi^t),
\end{equation}
where $\gamma > 0$ is the risk aversion parameter and $\varphi(T - t) > 0$ is the non-exponential discounting function. The wealth process is still given by \eqref{eq:wealth}. The following existence result is directly from Theorem \ref{thm:exist} and Lemma \ref{lem:cont}.
\begin{proposition}
	Suppose the random variables $Z_t$ in \eqref{eq:wealth} satisfy Condition (1) in Lemma \ref{lem:cont}. Moreover, assume the control $ 0 < \underline{u} \leq u_t \leq \bar{u} < \infty$ is in a compact interval. Then there exists a nonrandomized Markov equilibrium policy for \eqref{example4}.
\end{proposition}

\section*{Acknowledgment}
The authors thank the anonymous referees and editors for their valuable comments and suggestions that have greatly improved the manuscript.


\appendix
\section{Proof of Lemma \ref{lem:cont}} 
	Fix $x_t \in \R$. Consider a sequence $u^k$ converging to $u$. For notational simplicity, denote $\mu^k := \mu(x_t, u^k)$, $\mu := \mu(x_t, u)$, $\sigma^k := \sigma(x_t, u^k)$, and $\sigma := \sigma(x_t, u)$. Under condition (2), we have $\mu^k \rightarrow \mu$ and $\sigma^k \rightarrow \sigma$, when $k \rightarrow \infty$.
	
	A change of variables gives
	\begin{align*}
		\int V(x_{t+1}) P_t(dx_{t+1}| x_t, u^k) = \int^\infty_{-\infty} V(\mu^k + \sigma^k w) p(w) dw = \int^\infty_{-\infty} V(z) p\Big( \frac{z - \mu^k}{\sigma^k} \Big) \frac{dz}{\sigma^k}.
	\end{align*}
	Then
	\begin{align}
		&\left| \int V(x_{t+1}) P_t(dx_{t+1}| x_t, u^k) - \int V(x_{t+1}) P_t(dx_{t+1}| x_t, u) \right| \nonumber \\
		&\qquad \leq  \left| \int^\infty_{-\infty} V(z) p\Big( \frac{z - \mu^k}{\sigma^k} \Big) \frac{dz}{\sigma^k} - \int^\infty_{-\infty} V(z) p\Big( \frac{z - \mu}{\sigma} \Big) \frac{dz}{\sigma}\right| \nonumber \\
		& \qquad \leq M  \int^\infty_{-\infty}  \left| p\Big( \frac{z - \mu^k}{\sigma^k} \Big) \frac{1}{\sigma^k} - p\Big( \frac{z - \mu}{\sigma} \Big)  \frac{1}{\sigma} \right|dz, \label{eq1}
	\end{align}
	where the second inequality uses the boundedness of $|V| \leq M$. Denote $y = (z - \mu)/\sigma$, then
	\begin{align*}
		\int^\infty_{-\infty}  \left| p\Big( \frac{z - \mu^k}{\sigma^k} \Big) \frac{1}{\sigma^k} - p\Big( \frac{z - \mu}{\sigma} \Big)  \frac{1}{\sigma} \right|dz = \int^\infty_{-\infty}  \left| p\Big( \frac{\sigma y + \mu - \mu^k}{\sigma^k} \Big) \frac{1}{\sigma^k} - p(y)  \frac{1}{\sigma} \right| \sigma dy.
	\end{align*}
	Note that $\sigma > 0$ and can be moved into the absolute value. 
	
	For simplicity, denote $\tilde{\sigma}^k := \sigma/\sigma^k$ and $\tilde{\mu}^k := (\mu - \mu^k)/\sigma^k$. For given $0 < \delta < 1$ and $\delta' >0$, choose an integer $k$ large enough such that
	\begin{align*}
		|\tilde{\sigma}^k - 1| \leq \delta, \quad |\tilde{\mu}^k| \leq \delta'. 
	\end{align*}
	For a fixed $\varepsilon > 0$, consider $K > 0$ such that $\int^K_{-K} p(y) dy \geq 1 - \varepsilon/8$. Denote a constant $K^* := (K + \delta')/(1 - \delta)$. We have
	\begin{align*}
		\int^\infty_{K^*}  \left| p\Big( \frac{\sigma y + \mu - \mu^k}{\sigma^k} \Big) \frac{1}{\sigma^k} - p(y)  \frac{1}{\sigma} \right| \sigma dy & \leq \int^\infty_{K^*} p(\tilde{\sigma}^k y + \tilde{\mu}^k) \tilde{\sigma}^k dy + \int^\infty_{K^*} p(y) dy \\
		& \leq \int^\infty_{\tilde{\sigma}^k K^* + \tilde{\mu}^k} p(y') dy' + \int^\infty_{K^*} p(y) dy \\
		& \leq 2 \int^\infty_{K} p(y) dy,
	\end{align*}
	where the last inequality used the fact that $\max\{\tilde{\sigma}^k K^* + \tilde{\mu}^k, K^*\} \geq (1 - \delta) K^* - \delta' = K$, by the definition of $K^*$. Similarly, we have
	\begin{align*}
		\int^{-K^*}_{-\infty}  \left| p\Big( \frac{\sigma y + \mu - \mu^k}{\sigma^k} \Big) \frac{1}{\sigma^k} - p(y)  \frac{1}{\sigma} \right| \sigma dy
		& \leq 2 \int^{-K}_{-\infty} p(y) dy.
	\end{align*}
	Hence, noting the property of $K$, we obtain
	\begin{align}
		\int^\infty_{-\infty}  \left| p\Big( \frac{z - \mu^k}{\sigma^k} \Big) \frac{1}{\sigma^k} - p\Big( \frac{z - \mu}{\sigma} \Big)  \frac{1}{\sigma} \right|dz  \leq \int^{K^*}_{-K^*} \Big|p(\tilde{\sigma}^k y + \tilde{\mu}^k) \tilde{\sigma}^k - p(y) \Big| dy + \frac{\varepsilon}{4}. \label{eq2}
	\end{align}
	
	Next, we choose a constant $G > 0$, such that $\int^\infty_{-\infty} p(y) \id_{ \{p(y) > G \} } dy \leq \varepsilon/8$. Define 
	\begin{equation*}
		p_G(y) = p(y) \id_{\{p(y) \leq G \}} + G \id_{\{p(y) > G \}}, 
	\end{equation*}
	which is bounded. It yields
	\begin{align}
		& \int^{K^*}_{-K^*} \Big|p(\tilde{\sigma}^k y + \tilde{\mu}^k) \tilde{\sigma}^k - p(y) \Big| dy - \int^{K^*}_{-K^*} \Big|p_G(\tilde{\sigma}^k y + \tilde{\mu}^k) \tilde{\sigma}^k - p_G(y) \Big| dy \nonumber \\
		& \qquad \leq \int^{K^*}_{-K^*} \Big|p(\tilde{\sigma}^k y + \tilde{\mu}^k) \tilde{\sigma}^k - p_G(\tilde{\sigma}^k y + \tilde{\mu}^k) \tilde{\sigma}^k \Big| dy + \int^{K^*}_{-K^*} \Big|p(y) - p_G(y) \Big| dy \nonumber \\
		& \qquad \leq 2 \int^\infty_{-\infty} \Big|p(y) - p_G(y) \Big| dy = 2 \int^\infty_{-\infty} (p(y) - G) \id_{\{p(y) > G \}} dy \nonumber \\
		& \qquad \leq  2 \int^\infty_{-\infty} p(y) \id_{\{p(y) > G \}} dy \leq \varepsilon/4. \label{eq3}
	\end{align} 
	Finally, we apply Lusin's theorem to $p(y)$, which is Borel measurable on $\R$. There exists a continuous function $q(y)$, such that the Lebesgue measure of the set $\{p(y) \neq q(y)\}$ is not greater than $\varepsilon/(8G)$. Define
	\begin{equation*}
		q_G(y) = q(y) \id_{\{ 0 \leq q(y) \leq G \}} + G \id_{ \{ q(y) > G \} }.
	\end{equation*}
	Obviously, if $p(y) = q(y)$, then $p_G(y) = q_G(y)$. It implies that $\{p_G(y) \neq q_G(y) \} \subseteq \{ p(y) \neq q(y) \}$. As a consequence, 
	\begin{align}
		& \int^{K^*}_{-K^*} \Big|p_G(\tilde{\sigma}^k y + \tilde{\mu}^k) \tilde{\sigma}^k - p_G(y) \Big| dy \nonumber \\
		& \qquad \leq \int^{K^*}_{-K^*} \Big|q_G(\tilde{\sigma}^k y + \tilde{\mu}^k) \tilde{\sigma}^k - q_G(y) \Big| dy + \int^{K^*}_{-K^*} \Big|p_G(\tilde{\sigma}^k y + \tilde{\mu}^k) \tilde{\sigma}^k - q_G(\tilde{\sigma}^k y + \tilde{\mu}^k) \tilde{\sigma}^k \Big| dy \nonumber \\
		& \qquad \quad + \int^{K^*}_{-K^*} \Big|p_G(y) - q_G(y) \Big| dy \nonumber \\
		& \qquad \leq \int^{K^*}_{-K^*} \Big|q_G(\tilde{\sigma}^k y + \tilde{\mu}^k) \tilde{\sigma}^k - q_G(y) \Big| dy + 2 \int^\infty_{-\infty} \Big|p_G(y) - q_G(y) \Big| dy \nonumber \\
		& \qquad \leq \int^{K^*}_{-K^*} \Big|q_G(\tilde{\sigma}^k y + \tilde{\mu}^k) \tilde{\sigma}^k - q_G(y) \Big| dy + 2 \int^\infty_{-\infty} G \id_{ \{p_G(y) \neq q_G(y)\} } dy \nonumber \\
		& \qquad \leq \int^{K^*}_{-K^*} \Big|q_G(\tilde{\sigma}^k y + \tilde{\mu}^k) \tilde{\sigma}^k - q_G(y) \Big| dy + \varepsilon/4. \label{eq4}
	\end{align} 
	We can regard $q_G(a y + b) a$ as a function of three variables $(a, b, y)$. Clearly,  $q_G(a y + b) a$ is jointly continuous in $(a, b, y)$, since $q_G(\cdot)$ is continuous. By the uniform continuity on the compact set $[1 - \delta, 1 + \delta] \times [-\delta', \delta'] \times [-K^*, K^*]$, there exists an integer $N>0$ such that
	\begin{equation}\label{eq5}
		\sup_{k \geq N, y \in [-K^*, K^*]} \Big|q_G(\tilde{\sigma}^k y + \tilde{\mu}^k) \tilde{\sigma}^k - q_G(y) \Big| \leq \varepsilon/(8K^*).
	\end{equation}
	Combining \eqref{eq1}, \eqref{eq2}, \eqref{eq3}, \eqref{eq4}, and \eqref{eq5}, we obtain
	\begin{equation*}
		\left| \int V(x_{t+1}) P_t(dx_{t+1}| x_t, u^k) - \int V(x_{t+1}) P_t(dx_{t+1}| x_t, u) \right| \leq M \varepsilon. 
	\end{equation*}
	Since $\varepsilon > 0$ is arbitrary, the claim follows. 
\qed

\end{document}